\documentclass[12pt]{article}
\usepackage{amssymb,amsfonts,amsmath,amsthm,amsrefs,hyperref}
\usepackage{geometry}
\newcommand{\Q}{\mathbb{Q}}
\newcommand{\R}{\mathbb{R}}
\newcommand{\C}{\mathbb{C}}
\newcommand{\Cp}{\C^{\times}}
\newcommand{\N}{\mathbb{N}}

\newcommand{\T}{\mathbb{T}}

\newcommand{\8}{\infty}

\newcommand{\spa}{\mathrm{span}}

\newcommand{\Co}{\mathcal{C}}

\newcommand{\Ho}{\mathcal{H}}
\newcommand{\Int}{\mathrm{int}}

\newcounter{dummy} \numberwithin{dummy}{section}
\newtheorem{theorem}[dummy]{Theorem}
\newtheorem{lemma}[dummy]{Lemma}

\newtheorem{proposition}[dummy]{Proposition}
\newtheorem{corollary}[dummy]{Corollary}
\newtheorem{question}[dummy]{Question}
\theoremstyle{remark}
\newtheorem{remark}[dummy]{Remark}
\newtheorem{example}[dummy]{Example}

\begin{document}

\title{Invariant Hermitean Metric Generated By Reproducing Kernel Hilbert Spaces (Preliminary Version)}
\author{Eugene Bilokopytov\footnote{Email address bilokopi@myumanitoba.ca.}}
\maketitle

\begin{abstract}
For a Reproducing Kernel Hilbert Space on a complex domain we give a formula that describes the Hermitean metrics on the domain which are pull-backs of some metric on the (dual of) the RKHS via the evaluation map. Then we consider the question when such metrics are invariant with respect to the group of automorphisms of the domain. First we approach the problem by considering a stronger property, demanding that the original metric on the (dual of) the RKHS is invariant with respect to all (adjoints of) composition operators, induced by  automorphisms. However, we show that only the trivial metric satisfies this property. Then we characterise RKHS's for which the Bergman metric analogue studied in \cite{cd} and \cite{arsw} is automorphism-invariant.

\emph{Keywords:} Reproducing Kernel Hilbert Spaces; Hermitean metrics; unitary-invariance; automorphism-invariance;

MSC2010 32A36, 46E22, 47B32, 53B35.
\end{abstract}

\section{Introduction}

Geometric objects are largely characterized by their symmetries. Moreover, by geometry of an object we mean the prescribed class of symmetries. For example, the symmetries of Hilbert spaces are the unitaries, the symmetries of complex manifolds are the biholomorphisms and the symmetries of Hermitean manifolds are biholomorphims that are also isometries. In this article we will establish a connection between these three classes. In particular, we will consider a classic problem of endowing a given complex manifold with a Hermitean or Kahler metric, such that any automorphism of the manifold is automatically an isometry with respect to the chosen metric.

A well-known construction of such \emph{invariant} metric for bounded domains in $\C^{n}$ is due to Stefan Bergman (see \cite{berg}). An intermediate tool for doing that was the \emph{Bergman space}, which is the space of all square-integrable holomorphic functions on the domain. Note that the latter space is not an invariant object, since it depends on the Lebesgue measure, which is not biholomorphism-invariant. Later, Kobayashi in his celebrated paper \cite{kob} extended the definition of Bergman metric to complex manifolds. The construction was essentially the same, but the space of functions was replaced with the space of differential forms, which only depends on the holomorphic structure of the manifold, and so is automorphism-invariant. Kobayashi has also introduced another invariant distance, which is simpler to calculate explicitly. The latter distance received more attention in the works of Skwarczy\'nski (\cite{sk}). Meanwhile, Bergman's approach towards spaces of functions contributed to the early development of a new branch of Analysis - \emph{Reproducing Kernel Hilbert Space} (RKHS) Theory.

It is possible to replace the Bergman space in the aforementioned constructions with an arbitrary RKHS. In \cite{cd} this idea was used to construct a unitary invariant of multiplication with the free variable on the domains in $\C$ (see also \cite{am}). In \cite{arsw} Arcozzi, Rochberg, Sawyer and Wick did a thorough analysis of the metrics that were obtained this way. Rochberg then was able (see \cite{rochberg}) to solve a certain embedability problem using these metrics. However the applicability of his approach is limited, since the metric construction only "matches" some of the RKHS's. To overcome this obstacle one can try to find a "matching" metric for any RKHS.

One of the principal results in RKHS theory can be interpreted in the following way: any abstract function of two variables that has certain property of an inner product induces an embedding of that set into a Hilbert space. The embedding is unique up to unitary equivalence. Hence, if our set has a structure of a smooth manifold, we can pull back a unitary-invariant metric on a Hilbert space in a unique way through that embedding. We find a formula for this pull-back in terms of the kernel of the space (Theorem \ref{main}) and discuss the correspondence between the spaces and the metrics. Namely, we study possibilities of constructing an automorphism-invariant metric on a complex domain "from" a given RKHS.

Let us describe the content of the paper. Section \ref{prel} consists of preliminary material, including some notations concerning abstract functions of two variables, necessary basics of RKHS theory, geometry of complex domains and rigidity of sesqui-holomorphic functions. It also contains several results on unitary-invariant Hermitean metrics on Hilbert spaces taken from \cite{erz}, which we will adapt to the case when the Hilbert space is a dual of a RKHS in the next section.

Section \ref{pull} contains our main result, Theorem \ref{main}, which gives an analytic description of all Hermitean metrics on a domain in $\C^{n}$ which are pull-backs of unitary-invariant metrics from (the duals of) RKHS's. The section proceeds with the first attempt of finding an automorphism-invariant Hermitean metric among these metrics. We obtain a negative result, Proposition \ref{no}, which says that we cannot expect a unitary-invariant Hermitean metric on the dual of a RKHS to be invariant with respect to all composition operators induced by automorphisms.

The metrics used in \cite{cd} and \cite{arsw} appear in Section \ref{pull} as an example and are the main subject in Section \ref{proj}. It was stated in \cite{arsw} without a proof that these metrics determine the RKHS up to a "rescalling", and we give a proof of this fact using the rigidity result mentioned above. As a consequence we get that these metrics are automorphism-invariant if and only if the reproducing kernel under the consideration is projectively invariant with respect to the action of the group of automorphisms of the domain. This circumstance provides a connection between the present work and Operator theory and Representation theory. Namely, we give a description of projectively invariant kernels in Theorem \ref{prig} and discuss operations with such kernels in the next proposition. Proposition \ref{mrig} tells that a projectively invariant kernel is "almost" determined by its multipliers. This result was taken from \cite{bb}, but there it is established for Bergman-like spaces. The section concludes with an application of this proposition to weighted Bergman spaces.

The Section \ref{sup} is of complementary nature. First we discuss possible generalizations of Lemma \ref{no1} and Proposition \ref{mrig} and pose some questions related to these results. Further, we present the proof of Proposition \ref{wu} and conclude with a version of Theorem \ref{main} for the field of real numbers.

\section{Preliminaries}\label{prel}

Let $\Cp=\C\backslash\left\{0\right\}$ and let $\T=\left\{\lambda\in\C,\left|\lambda\right|=1\right\}$. We start with some notations and conventions regarding (abstract) functions of two variables. Let $X$ be a set. Consider the simplest functions of two variables: if $\omega,\upsilon:X\to\C$, define $\omega\otimes\upsilon:X\times X\to \C$ by $\left[\omega\otimes\upsilon\right]\left(y,x\right)=\omega\left(x\right)\upsilon\left(y\right)$. On the other hand, consider the transition from a function of two variables, to the function of one variable: if $L:X\times X\to \C$, define the \emph{diagonal function} $\widehat{L}:X\to \C$ by $\widehat{L}\left(x\right)=L\left(x,x\right)$, for $x\in X$. In fact, $\widehat{L}=L\circ j_{X}$, where $j_{X}:X\to X\times X$ is defined by $j_{X}\left(x\right)=\left(x, x\right)$. If $\Phi:X\to X$ define $L\circ\Phi:X\times X\to \C$ by $\left[L\circ\Phi\right]\left(x,y\right)=L\left(\Phi\left(x\right),\Phi\left(y\right)\right)$, for $x,y\in X$. Note that $\widehat{L\circ\Phi}=\widehat{L}\circ\Phi$. We will say that $L$ is \emph{non-degenerate} if $\widehat{L}\left(x\right)>0$ for all $x\in X$. This term will be justified in the next subsection.

We will say that a non-degenerate function $M:X\times X\to \C$ is a \emph{rescalling} of a non-degenerate function $L:X\times X\to \C$ if there is a function $\omega:X\to\C$, such that $M=\omega\otimes\overline{\omega}L$. In this case we will also say that $\omega$ \emph{rescales} $L$ to $M$. Note that $\lambda\omega$ also rescales $L$ to $M$ for any $\lambda\in\T$. Clearly $\left|\omega\right|^{2}=\frac{\widehat{M}}{\widehat{L}}$, and so the relation of being a rescalling is symmetric; in fact this is an equivalence relation. Hence, we will say that $L$ and $M$ are rescallings if one of them is a rescalling of another. Finally note that if $L_{1},M_{1}$ and $L_{2},M_{2}$ are rescallings, then $L_{1}L_{2}$ and $M_{1}M_{2}$ are also rescallings.\\

\textbf{RKHS.} Let $H$ be an inner product space and let $h\in H$. By $h^{*}$ we will denote the continuous linear functional on $H$ that acts by $\left<g,h^{*}\right>=\left<g,h\right>_{H}$. The correspondence $*$ is an antilinear isometry from $H$ into its dual $H^{*}$. We can define an inner product on $H^{*}$ as the continuous extension of the push-forward $\left<g^{*},h^{*}\right>_{H^{*}}=\left<h,g\right>_{H}$; this inner product agrees with the natural norm on $H^{*}$. If we further assume the completeness of $H$ i.e. that $H$ is a Hilbert space, then Riesz theorem ensures that $*$ is a bijection, and so it has an inverse which is also an antilinear isometry. Let $^{*}l$ be the element of $H$, such that $\left(^{*}l\right)^{*}=l$, for $l\in H^{*}$.

Let $X$ be a set and let $F$ be a linear subspace of the space of all complex-valued functions on $X$, equipped with an inner product, turning $F$ into a Hilbert space. Then, $F$ is called a \emph{Reproducing Kernel Hilbert Space (RKHS)} over $X$ if there is a function (\emph{reproducing kernel}) $K:X\times X\to \C$, such that $K_{x}=K\left(\cdot,x\right)\in F$, for each $x\in X$, and $f\left(x\right)=\left<f,K_{x}\right>_{F}$, for each $f\in F$.

Defining $\kappa\left(x\right)=K_{x}^{*}$, we can see that $F$ is a RKHS if and only if there is an \emph{evaluation map} $\kappa:X\to F^{*}$, such that $\overline{\spa~\kappa\left(X\right)}=F^{*}$ and $f\left(x\right)=\left<f,\kappa\left(x\right)\right>$, for each $x\in X$ and $f\in F$; moreover,  $K\left(x,y\right)=\left<\kappa\left(x\right),\kappa\left(y\right)\right>_{F^{*}}$. Note that $K$ is conjugate-symmetric, i.e. $K\left(x,y\right)=\overline{K\left(y,x\right)}$, for any $x,y\in X$; also\\ $\widehat{K}\left(x\right)=\|\kappa\left(x\right)\|^{2}\ge 0$, for any $x\in X$, and the equality is attained precisely at the points of $X$ where all elements of $F$ vanish. Hence, $K$ is non-degenerate if and only if there is no null point evaluations.

A \emph{positive semi-definite} function (kernel) on $X$ is a function $K:X\times X\to \C$ such that for any $x_{1},...,x_{n}\in X$, the matrix $\left(K\left(x_{i},x_{j}\right)\right)_{i,j=1}^{n}$ is positive semi-definite. It is easy to see that pointwise limits of positive semi-definite kernels are positive semi-definite and, by Schur's Theorem, a product of positive semi-definite kernels is positive semi-definite. In particular, if $K$ is positive semi-definite, then so is $K^{n}$, for any $n\in\N$. Note however, that $K^{s}$ is not necessarily positive semi-definite if $s\not\in\N$. It is easy to see that if $\omega:X\to \C$, then $\omega\otimes\overline{\omega}$ is positive semi-definite. Thus, any rescalling of a positive semi-definite kernel is positive semi-definite. The importance of the established class of functions is revealed by the following theorem (see \cite{aron}).

\begin{theorem}[Moore-Aronszajn]
Let $X$ be a set. A function\\ $K:X\times X\to \C$ is positive semi-definite if and only if there is a (unique) RKHS over $X$ with reproducing kernel equal to $K$.
\end{theorem}

The word "unique" in the theorem means that any two RKHS's with the same kernel coincide as sets and have the same Hilbert Space structure. Consequently, one of the ways to interpret Moore-Aronszajn theorem is: for any positive semi-definite kernel $K$ on a set $X$ there is a Hilbert space $F$ and a map $\kappa:X\to F^{*}$, such that $\kappa\left(X\right)$ spans $F^{*}$ and $K$ is represented by a restriction of the scalar product of $F^{*}$ on $\kappa\left(X\right)$. Moreover, the pair $\left(F,\kappa\right)$ is unique up to a unitary equivalence.\\

\textbf{Complex domains.} Let $X$ be a domain in $\C^{n}$, i.e. an open connected set. An \emph{automorphism} of $X$ is a biholomorphism of $X$ into itself, i.e. a bijective holomorphic self-map of $X$. We will denote the group of all automorphisms of $X$ by $Aut\left(X\right)$. We say that $X$ is \emph{homogeneous} if $Aut\left(X\right)$ acts transitively on $X$, i.e. there is $x\in X$ such that the orbit $Aut\left(X\right)\cdot x=X$. Consider the following strengthenings of homogeneity: $X$ is called \emph{weakly symmetric} if for any $x,y\in X$ there is $\Phi\in Aut\left(X\right)$, such that $\Phi\left(x\right)=y$ and $\Phi\left(y\right)=x$; $X$ is called \emph{symmetric} if $\Phi$ in the previous definition is required to be an involution. It is a classic result of Cartan that in low dimensions these concepts coincide, but it is not true in general.

Let $\Ho\left(X\right)$ be the linear space of all holomorphic functions on $X$ endowed with the compact-open topology. A \emph{tangent vector} at $x\in X$ is an expression of the form $\sum\limits_{j=1}^{n}a_{j}\frac{\partial}{\partial z_{j}}\left|_{x}\right.$, where $a_{j}\in\C$, for $j\in\overline{1,n}$. The \emph{tangent space} $T_{x}=T_{x}^{1,0}X$ is the the vector space of all tangent vectors at $x$. If  $u=\sum\limits_{j=1}^{n}a_{j}\frac{\partial}{\partial z_{j}}\left|_{x}\right.$ and  $v=\sum\limits_{j=1}^{n}b_{j}\frac{\partial}{\partial z_{j}}\left|_{x}\right.$ denote $\overline{v}=\sum\limits_{j=1}^{n}\overline{b_{j}}\frac{\partial}{\partial \overline{z_{j}}}\left|_{x}\right.$, which is a Wirtinger derivative, and $u\otimes \overline{v}=\sum\limits_{j,k=1}^{n}a_{j}\overline{b_{k}}\frac{\partial^{2}}{\partial z_{j}\partial\overline{z_{k}}}\left|_{x}\right.$ - a "directional Laplacian".\\

\textbf{Sesqui-holomorphic Functions.} Let $X$ be a domain in $\C^{n}$. A sesqui-holomorphic function on $X$ is a function $L:X\times X\to \C$ which is holomorphic in the first variable and antiholomorphic in the second. Let $X^{*}=\left\{\overline{x}\left|x\in X\right.\right\}\subset \C^{n}$. By Hartogs theorem $L$ is sesqui-holomorphic if and only if the function $L':X\times X^{*}\to \C$ defined by $L'\left(x,y\right)=L\left(x,\overline{y}\right)$ is holomorphic. Hence, any sesqui-holomorphic function is continuous and if it vanishes on an open set in $X\times X$, it vanishes everywhere. The following result, which is an immediate consequence of \cite[Theorem II.4.7]{bm} shows that such functions are completely determined by its values on the diagonal (see also \cite{nik}).

\begin{proposition}\label{rig}
Let $L$ and $M$ be sesqui-holomorphic functions on $X$. If $\widehat{L}=\widehat{M}$, then $L=M$.
\end{proposition}

It is clear that if we only assume that $\widehat{L}=\widehat{M}$ on an open set in $X$, we also get that $L=M$. The following fact is of a similar spirit (see {\cite[Theorem 6.7]{erz2}}).

\begin{theorem}\label{rrig}
Let $L$ and $M$ be sesqui-holomorphic non-degenerate functions on $X$. The following are equivalent:
\item[(i)] $L$ and $M$ are rescallings;
\item[(ii)] $\frac{\partial^{2}}{\partial z_{j}\partial \overline{z_{k}}}\log \widehat{L}=\frac{\partial^{2}}{\partial z_{j}\partial \overline{z_{k}}}\log \widehat{M}$ on an open set in $X$;
\item[(iii)] There is $y\in X$ and an open $U\subset X$ such that $\frac{\left|L\left(x,y\right)\right|^{2}}{L\left(x,x\right)L\left(y,y\right)}=\frac{\left|M\left(x,y\right)\right|^{2}}{M\left(x,x\right)M\left(y,y\right)}$, for every $x\in U$.
\end{theorem}

Note that the local equivalence of (i) and (ii) is known as Calabi Rigidity (see \cite{cal}).

\begin{remark}\label{rss} It is easy to see that any function that rescales sesqui-holomorphic functions is holomorphic. Moreover, it is determined uniquely up to a constant unimodular multiple.\qed\end{remark}

We also need to calculate derivatives of the diagonals of sesqui-holomorphic functions. For simplicity, assume that $L:\C\times\C\to\C$. Recall that $\widehat{L}=L\circ j_{\C}$, where $j_{\C}=Id\oplus Id$. Note that $\frac{d}{dz}j_{\C}=1\oplus 1$, $\frac{d}{d\overline{z}}j_{\C}=0\oplus 0$ and $\frac{\partial}{\partial w}L=\frac{\partial}{\partial\overline{z}}L=0$. Then using the chain rule for the Wirtinger derivatives we get that $\frac{d}{dz}\widehat{L}=\left(\frac{\partial}{\partial z}L\right)\circ j_{\C}=\widehat{\frac{\partial}{\partial z}L}$ and $\frac{d}{d\overline{z}}\widehat{L}=\widehat{\frac{\partial}{\partial \overline{w}}L}$. Since $\frac{\partial}{\partial z}L$ and $\frac{\partial}{\partial \overline{w}}L$ are also sesqui-holomorphic, it follows that $\frac{d^{l+m}}{dz^{l}d\overline{z}^{n}}\widehat{L}=\widehat{\frac{\partial^{l+m}}{\partial z^{l}\partial \overline{w}^{m}}L}$. Finally, if $\widehat{L}$ is real-valued, then $\frac{d}{d\overline{z}}\widehat{L}=\overline{\frac{d}{dz}\widehat{L}}$.

From the formulas above it is easy to deduce that if $X$ is a domain in $\C^{n}$, $x\in X$, $u,v\in T_{x}$ and $L$ is sesqui-holomorphic on $X$, then $u\widehat{L}=uL\left(\cdot,x\right)$ and $u\otimes \overline{v}\widehat{L}=u\overline{v}_{2}L=\overline{v}u_{1}L$, where the subindex indicates the variable with respect to which the differentiation is implemented.\\

\textbf{Sesqui-holomorphic Reproducing Kernels.} If $L$ is conjugate-symmetric it is sesqui-holomorphic if and only if it is holomorphic in the first variable. On the other hand, sesqui-holomorphic function $L$ is conjugate-symmetric if and only if $\widehat{L}$ is real valued. While necessity is obvious, sufficiency follows from Proposition \ref{rig} applied to $L$ and $L^{*}$ which is a sesqui-holomorphic function on $X$ defined by $L^{*}\left(x,y\right)=\overline{L\left(y,x\right)}$. If we then assume that $\widehat{L}$ does not vanish, then either $L$ or $-L$ is non-degenerate. Let us  furthermore consider positive semi-definite sesqui-holomorphic functions. In fact, a sesqui-holomorphic function is positive semi-definite if it is positive semi-definite on an open set in $X$ (see the proof in \cite{bbb}). The following theorem describes what kind of RKHS's we obtain from such kernels. See the proof in \cite{gl}.

\begin{theorem}\label{shrkhs}  Let $X$ be a domain in $\C^{n}$, let $K$ be a positive semi-definite kernel on $X$ and let $F$ be the corresponding RKHS. The following are equivalent:
\item[(i)] $K$ is a sesqui-holomorphic kernel;
\item[(ii)] $F$ consists of holomorphic functions;
\item[(iii)] The evaluation map $\kappa$ is holomorphic into $F^{*}$.
\end{theorem}

The simplest examples of positive semi-definite sesqui-holomorphic kernels are functions of the form $\omega\otimes\overline{\omega}$, where $\omega\in\Ho\left(X\right)$. Also note that there are non-degenerate sesqui-holomorphic functions, which are not positive semi-definite. For example, $L\left(x,y\right)=1-x\overline{y}$ is not positive semi-definite on the unit disk, since $1=\left|L\left(0,\frac{1}{2}\right)\right|^{2}\not\le L\left(0,0\right)L\left(\frac{1}{2},\frac{1}{2}\right)=\frac{3}{4}$.\\

\textbf{Unitary Invariant Hermitean Metrics.} Let $H$ be a complex Hilbert space, let $G=H\backslash\left\{0\right\}$ and let $\sigma:G\times H\times H\to\C$ be such that $\sigma_{g}=\sigma\left(g,\cdot,\cdot\right)$ is conjugate-symmetric and sesquilinear on $H$, for every $g\in G$. Then $\sigma$ is called positive (semi-)definite if $\sigma_{g}$ is positive (semi-)definite for each $g\in G$. We say that $\sigma$ is invariant with respect to an injective linear operator $T$ on $H$ if $\sigma_{Tg}\left(Tf,Th\right)=\sigma_{g}\left(f,h\right)$ for every $g\in G$ and $f,h\in H$. In this case the length defined by $\sigma$ is also invariant with respect to $T$.

The following theorem incorporates several results from \cite{erz}: Proposition 2.8, Corollary 3.4, Theorem 3.7 and Remark 3.8. We characterise unitary-invariant and congruency-invariant Hermitean metric on $G$, i.e. such that are invariant with respect to all unitaries and congruencies (scalar multiples of isometries). The importance of the latter invariance is justified by part (iv) of the theorem.

\begin{theorem}\label{rm1} The function $\sigma$ is unitary-invariant if and only if there are (unique) functions $\varphi,\psi:\left(0,+\8\right)\to\R$, such that $\sigma_{g}\left(f,h\right)=\varphi\left(\|g\|^{2}\right)\left<f,h\right>+\psi\left(\|g\|^{2}\right)\left<f,g\right>\left<g,h\right>$, for $g\in G$, and $f,h\in H$. Moreover, in this case the following hold:
\item[(i)] $\sigma$ is positive definite if and only if $\varphi\left(r\right)> 0$ and $\varphi\left(r\right)+r\psi\left(r\right)> 0$ for every $r>0$;
\item[(ii)] The degree of smoothness of $\sigma$ coincides with the minimal degree of smoothness of $\varphi$ and $\psi$;
\item[(iii)] $\sigma$ is invariant with respect to all congruencies if and only if there are (unique) $a,b\in\R$, such that $\varphi\left(t\right)=\frac{a}{t}$ and $\psi\left(t\right)=\frac{b}{t^{2}}$;
\item[(iv)] If $\sigma$ is invariant with respect to a linear operator $T:H\to H$, then $T$ is a congruency.
\end{theorem}

We supply the theorem with several of remarks and an example also taken from \cite{erz}. In the following section they will be adapted to the case when $H^{*}$ is a RKHS, as well as the theorem above.

\begin{remark}\label{k}
Once $\varphi$ and $\psi$ are sufficiently smooth, $\varphi\left(r\right)> 0$ and $\varphi\left(r\right)+r\psi\left(r\right)> 0$, then $\sigma$ is a Hermitean metric on $G$. Furthermore, $\sigma$ is a Kaehler metric if and only if $\psi=\varphi'$. In this case $\chi\circ\|\cdot\|^2$ is the potential of this metric, where $\chi'=\varphi$.\qed
\end{remark}

\begin{remark}\label{dg2}
The non-strict analogues of the strict inequalities in part (i) correspond to the positive semi-definiteness of $\sigma$. In particular, if $\varphi\left(r\right)=-r\psi\left(r\right)$, for every $r>0$, then $\sigma$ glues elements that are scalar multiples of each other, i.e. factorizes by $\C$-lines. The case when $\varphi\left(r\right)=0$ leads to identifying all elements of norm $r$.\qed
\end{remark}

\begin{remark}\label{dg1}
The positive definiteness of congruency-invariant $\sigma$ is equivalent to $a+b>0$ in part (iii). The latter contradicts to the necessary condition for $\sigma$ to be Kaehler, which is reduced to $a=-b$. Thus, there is no Kaehler metrics on $G$ invariant with respect to all congruencies.\qed
\end{remark}

\begin{example}\label{fs}
Let $\tilde{\sigma}_{g}\left(f,h\right)=\frac{1}{\|g\|^{2}}\left<f,h\right>-\frac{1}{\|g\|^{4}}\left<f,g\right>\left<g,h\right>$. By Remark \ref{dg1}, this is the unique (up to scalar multiplication) "degenerate Kaehler metric" which is invariant with respect to all congruencies. Using Remark \ref{dg2} one can show that it is also the unique (up to scalar multiplication) "degenerate Kaehler metrics" which factorizes by the $\C$-lines. Since $\tilde{\sigma}$ is the pull-back of the classical Fubiny-Study metric on the projective space $PH$ via the natural quotient map, we find it natural to call $\tilde{\sigma}$ the Fubini-Study metric on  $G$. Note, that $2\log\|\cdot\|$ is the potential of this metric. Inspired by \cite{arsw} and following \cite{kob}, we consider two congruency-invariant pseudodistances on $G$ related to $\tilde{\sigma}$. For $g,h\in G$ define $$\delta_{1}\left(g,h\right)=\sin\angle\left(g,h\right)=\sqrt{1-\frac{\left|\left<g,h\right>\right|^{2}}{\|g\|^{2}\|h\|^{2}}},~\delta_{2}\left(g,h\right)=\sin\frac{\angle\left(g,h\right)}{2}=\sqrt{2-2\frac{\left|\left<g,h\right>\right|}{\|g\|\|h\|}}.$$ Note that $\delta_{1}$ and $\delta_{2}$ also factorise by the $\C$-lines; while the geometric meaning of $\delta_{1}$ is obvious, $\delta_{2}$ is the distance between the intersections of the $\C$-lines defined by $g,h$ and the unit sphere. The importance of $\delta_{1}$ and $\delta_{2}$ for us is determined by the fact that the length of the curves with respect to $\delta_{1}$, $\delta_{2}$ and $\sigma$ coincide.\qed
\end{example}

\section{Pull-back of a unitary invariant Hermitean metric on a dual of a RKHS}\label{pull}

\textbf{Main Result.} Assume that $X$ is a domain in $\C^{n}$ and $F$ is a RKHS with a sesqui-holomorphic kernel $K$, which by Theorem \ref{shrkhs} is equivalent to holomorphicity of the evaluation map $\kappa$. Then any Hermitean metric on $F^{*}$ can be pulled back to $X$ via $\kappa$. Since there are a lot of possible representations of $K$ as a scalar product, but all of them are unitary equivalent, the most natural choice is to pull-back the unitary-invariant Hermitean metrics on $F^{*}$. In what follows we provide an analytic description for all such pull-backs.

Observe that all elements of $F$ are holomorphic, since any function $f\in F$ can be represented as a composition of $\kappa$ and a \textbf{linear} functional $f=f^{**}$ on $F^{*}$. Moreover, applying the chain rule to this composition, we get that $u f=\left<f,u\kappa\right>$, for any $u\in T_{x}$, $x\in X$. Here $u\kappa$ is the limit of elements of $F^{*}$, analogous to the limit of scalars, when $u$ is applied to usual functions. For example, if $X=\C$, $x=0$ and $u=\left.\frac{d}{dz}\right|_{0}$, then $u\kappa=\lim_{z\to 0}\frac{\kappa\left(z\right)-\kappa\left(0\right)}{z}$. We will call such functionals \emph{derivative point evaluations}. The following result characterizes the mutual location of the point evaluations and derivative point evaluations in $F^{*}$.

\begin{proposition}\label{dpe1}
Let $x,y,z\in X$, let $u\in T_{x}$ and let $v\in T_{y}$. Then
\item[(i)] $\left<u\kappa,\kappa\left(z\right)\right>_{F^{*}}=uK\left(\cdot,z\right)$ and $\left[^{*}u\kappa\right]\left(z\right)=\overline{uK\left(\cdot,z\right)}=\overline{u}K\left(z,\cdot\right)$;
\item[(ii)] $\left<u\kappa, v\kappa\right>_{F^{*}}=\overline{v}u_{1}K$, where the inner differentiation is with respect to the first variable.
\end{proposition}
\begin{proof} We will use properties of the operation $\ast$ between $F$ and $F^{*}$. Part (i) follows from $$\overline{\left[^{*}u\kappa\right]\left(z\right)}=\overline{\left<^{*}u\kappa,\kappa\left(z\right)\right>}=\left<u\kappa,\kappa\left(z\right)\right>_{F^{*}}=\left<u\kappa,^{*}\kappa\left(z\right)\right>=uK_{z}=uK\left(\cdot,z\right).$$
(ii): We can apply $v$ to the function $^{*}u\kappa\in F$, which was calculated in the previous part. We get
$v\overline{u_{1}K}=v\left(^{*}u\kappa\right)=\left<v\kappa, ^{*}u\kappa\right>=\left<v\kappa, u\kappa\right>_{F^{*}}$, and so\\ $\left<u\kappa, v\kappa\right>_{F^{*}}=\overline{v\overline{u_{1}K}}=\overline{v}u_{1}K$.
\end{proof}

Thus, the functions that represent the "derivative point evaluations" are the conjugates of the corresponding derivatives of $K$ in the first variable. From the preceding proposition and the formulas for derivatives of a diagonal of a sesqui-holomorphic functions we get the following corollary.

\begin{corollary}\label{dpe2}
Let $x\in X$ and let $u,v\in T_{x}$. Then $\left<u\kappa,\kappa\left(x\right)\right>_{F^{*}}=u\widehat{K}$ and $\left<u\kappa, v\kappa\right>_{F^{*}}=u\otimes\overline{v}\widehat{K}$.
\end{corollary}

Now we can express the pull-back of any unitary invariant Hermitean metric on $F^{*}\backslash\left\{0\right\}$ via $\kappa$.

\begin{theorem}\label{main}
Let $X$ be a domain in $\C^{n}$ and let $K$ be a sesqui-holomorphic non-degenerate positive semi-definite kernel with the corresponding RKHS $F$. Assume that a Hermitean metric $\sigma^{K}$ on $X$ is a pull-back of a certain unitary-invariant metric on $F^{*}\backslash\left\{0\right\}$. Then there are smooth functions $\varphi,\psi:\left(0,+\8\right)\to \R$, such that for any $x\in X$ and $u,v\in T_{x}$ we have
$$\sigma^{K}_{x}\left(u,v\right)=\varphi\left(\widehat{K}\left(x\right)\right)u\otimes\overline{v}\widehat{K}+\psi\left(\widehat{K}\left(x\right)\right)u\widehat{K}\overline{v}\widehat{K}.$$
\end{theorem}
\begin{proof}
Let $G=F^{*}\backslash\left\{0\right\}$, and let $\sigma:G\times F^{*}\times F^{*}\to\C$ be a smooth unitary-invariant function, such that $\sigma_{g}$ is conjugate-symmetric sesquilinear on $F^{*}$, for every $g\in G$. By virtue of Theorem \ref{rm1}, there are smooth real functions $\varphi,\psi$, such that $$\sigma_{g}\left(f,h\right)=\varphi\left(\|g\|^{2}\right)\left<f,h\right>+\psi\left(\|g\|^{2}\right)\left<f,g\right>\left<g,h\right>,$$ for $g\in G$, and $f,h\in F^{*}$. For tangent vectors $u,v$ at $x$ by definition\\ $\sigma^{K}_{x}\left(u,v\right)=\sigma_{\kappa\left(x\right)}\left(D\kappa_{x}u,D\kappa_{x}v\right)$, where $D\kappa_{x}$ is the differential of $\kappa$ at $x$. Since $X$ is a domain in $\C^{n}$ and we have identified the tangent space to $F^{*}$ with $F^{*}$, $D\kappa_{x}$ is just the "Jacobi matrix" of $\kappa$ at $x$, i.e. if $u=\sum\limits_{i=1}^{n}a_{i}\frac{\partial}{\partial z_{i}}\left|_{x}\right. $, then $D\kappa_{x}u=\sum\limits_{i=1}^{n}a_{i}\frac{\partial\kappa}{\partial z_{i}}\left(x\right) =u\kappa$. \footnote{In general, $D\kappa_{x}u$ is an element of the tangent space of $F^{*}$ at $\kappa\left(x\right)$ such that $\left(D\kappa_{x}u\right)\phi=u\left(\phi\circ\kappa\right)$, for any smooth scalar function $\phi$ on $G$. Consider $\phi=f^{**}$, which is a linear (and so smooth) function on $F^{*}$. Then $\left(D\kappa_{x}u\right)f^{**}=u\left(f^{**}\circ\kappa\right)=uf=\left<f,u\kappa\right>$. Thus, via identification of the tangent space of $F^{*}$ with $F^{*}$ and $f^{**}$ with $f$ we get that $D\kappa_{x}u=u\kappa$.} Hence, from the preceding corollary
\begin{eqnarray*}
\sigma^{K}_{x}\left(u,v\right)&=&\sigma_{\kappa\left(x\right)}\left(u\kappa,v\kappa\right)=\varphi\left(\|\kappa\left(x\right)\|^{2}\right)\left<u\kappa,v\kappa\right>+\psi\left(\|\kappa\left(x\right)\|^{2}\right)\left<u\kappa,\kappa\left(x\right)\right>\left<\kappa\left(x\right),v\kappa\right>\\
&=&\varphi\left(\widehat{K}\left(x\right)\right)u\otimes\overline{v}\widehat{K}+\psi\left(\widehat{K}\left(x\right)\right)u\widehat{K}\overline{v}\widehat{K}.
\end{eqnarray*}
\end{proof}

\begin{remark}In the local coordinates the metric is expressed as $$\sum\limits_{i,j=1}^{n}\left[\varphi\circ\widehat{K}\frac{\partial^{2}}{dz_{i}d\overline{z}_{j}}\widehat{K}+\psi\circ\widehat{K}\frac{\partial}{dz_{i}}\widehat{K}\frac{\partial}{d\overline{z}_{j}}\widehat{K}\right]dz_{i}\otimes d\overline{z}_{j}.$$\end{remark}

\begin{remark}\label{pd}
It follows from part (i) of Theorem \ref{rm1} that if $\varphi\left(r\right)> 0$ and $\varphi\left(r\right)+r\psi\left(r\right)> 0$ for every $r>0$ then $\sigma$ is positive definite. However, this condition is not necessary. In particular, if $\varphi\left(r\right)> 0$ and $\varphi\left(r\right)+r\psi\left(r\right)\ge 0$ for every $r>0$, then $\sigma$ is positive definite, \textbf{unless} there are $x\in X$ and $u\in T_{x}$ such that $u\kappa=\kappa\left(x\right)$, or equivalently, $K\left(\cdot,x\right)\in\spa \left\{\frac{\partial}{\partial\overline{w_{1}}}K\left(\cdot,x\right),...,\frac{\partial}{\partial\overline{w_{n}}}K\left(\cdot,x\right)\right\}$. Indeed, the only way $\sigma^{K}_{x}\left(v,v\right)=0$ is when $\|v\kappa\|^{2}=\frac{\left|\left<v\kappa,\kappa\left(x\right)\right>\right|^{2}}{\|\kappa\left(x\right)\|^{2}}$, which can only happen when there is $\lambda\in\C$, such that $\lambda\kappa\left(x\right)=v\kappa$. Thus, either $v=0$, or $u\kappa=\kappa\left(x\right)$, where $u=\lambda^{-1}v$.

However, this phenomenon usually does not occur to the spaces of interest.
\qed\end{remark}

\begin{remark} If $\varphi$ and $\psi$ are sufficiently smooth and $\psi=\varphi'$, then $\sigma^{K}$ is Kaehler. Using the chain rule one can verify that its potential is $\chi\circ\widehat{K}$, where $\chi'=\varphi$. If $\chi$ can be extended to a holomorphic function in a neighborhood of $\left(0,+\8\right)\subset\C$, then by Calabi rigidity (see \cite{cal}) we have that $\sigma^{K}=\sigma^{L}$ if and only if there is a holomorphic function $\theta$, such that $\chi\left(K\left(x,y\right)\right)=\chi\left(L\left(x,y\right)\right)+\theta\left(x\right)+\overline{\theta\left(y\right)}$, for every $x,y\in X$. \qed
\end{remark}

\begin{example}
Let us consider the pull-back of the most natural metric on a Hilbert space - the Euclidean metric. This metric corresponds to $\varphi\equiv 1$ and $\psi\equiv 0$. Then, $\sigma^{K}_{x}\left(u,u\right)=\|u\kappa\|^{2}=u\otimes\overline{u}\widehat{K}$, and so for any $f\in F$ we have that $$\left|uf\right|=\left|\left<f,u\kappa\right>\right|\le\|f\|\|u\kappa\|=\|f\|\sqrt{\sigma^{K}_{x}\left(u,u\right)}.$$ Thus, $f$ is a Lipschitz function on $X$ with respect to $\sigma^{K}$, with Lipschitz constant at most $\|f\|$.\qed
\end{example}

\begin{example}[\textbf{Bergman metric}] Let $X$ be a bounded domain in $\C^{n}$. The \emph{Bergman space} $A^{2} \left(X \right)$ is the space of square-integrable holomorphic functions, i.e. $$A^{2} \left(X \right)=\left\{f\in\Ho\left(X\right),\left\|f\right\|_{2}=\left(\int _{X}\left|f \right|^{2} d\mu  \right)^{\frac{1}{2} }<+\8 \right\},$$ where $\mu$ is the Lebesgue measure on $X$. It turns out that this space is a RKHS; its kernel $K$ is called the \emph{Bergman kernel} of $X$. The \emph{Bergman metric} of $X$ is the pull-back $\tilde{\sigma}^{K}$ of the Fubini-Study metric $\tilde{\sigma}$ (see Example \ref{fs}) as described above; its potential is $\log\circ\widehat{K}$.

If $x=\left[x_{1},x_{2},...,x_{n}\right]\in X$ and $u=\sum\limits_{i=1}^{n}a_{i}\frac{\partial}{\partial z_{i}}\left|_{x}\right.\ne 0 $, then $f$ defined by $f\left(z_{1},...,z_{n}\right)=\overline{a_{1}}\left(z_{1}-x_{1}\right)+...+\overline{a_{n}}\left(z_{n}-x_{n}\right)$ belongs to $A^{2} \left(X \right)$ and $f\left(x\right)=0$, while \\ $uf=\left|a_{1}\right|^{2}+...+\left|a_{n}\right|^{2}\ne 0$, and so $\kappa\left(x\right)\ne u\kappa$. Thus, by virtue of Remark \ref{pd}, the Bergman metric is positive definite.\qed
\end{example}

Let $K$ be a sesqui-holomorphic non-degenerate positive semi-definite kernel on $X$. Consider the pull-backs $\delta_{1}^{K}$ and $\delta_{2}^{K}$ of $\delta_{1}$ and $\delta_{2}$ (see Example \ref{fs}) through $\kappa$. It is easy to see that $$\delta_{1}^{K}\left(x,y\right)=\sqrt{1-\frac{\left|K\left(x,y\right)\right|^{2}}{K\left(x,x\right)K\left(y,y\right)}},~\delta_{2}^{K}\left(x,y\right)=\sqrt{2-2\frac{\left|K\left(x,y\right)\right|}{\sqrt{K\left(x,x\right)K\left(y,y\right)}}},$$
for any $x,y\in X$. Clearly, if there is no linearly dependant point evaluations, then both $\delta_{1}^{K}$ and $\delta_{2}^{K}$ are distances on $X$. Since the length with respect to $\tilde{\sigma}$, $\delta_{1}$ and $\delta_{2}$ coincide, the same is true for $\tilde{\sigma}^{K}$, $\delta_{1}^{K}$ and $\delta_{2}^{K}$. This makes $\delta_{1}^{K}$ and $\delta_{2}^{K}$ useful, since they are given explicitly, while the actual distance generated by $\tilde{\sigma}^{K}$ can be difficult to calculate. Let us establish the Calabi rigidity for $\tilde{\sigma}^{K}$, and also for $\delta_{1}^{K}$ and $\delta_{2}^{K}$. From Theorem \ref{rrig} we derive the following result.

\begin{corollary}\label{caldelt} Let $K$ and $L$ be sesqui-holomorphic non-degenerate positive semi-definite kernels on $X$. The following are equivalent:
\item[(i)] $K$ and $L$ are rescallings;
\item[(ii)] $\tilde{\sigma}^{K}=\tilde{\sigma}^{L}$;
\item[(iii)] $\delta_{1}^{K}=\delta_{1}^{L}$;
\item[(iv)] $\delta_{2}^{K}=\delta_{2}^{L}$.
\end{corollary}

\textbf{Automorphism-invariance.} The metrics $\tilde{\sigma}^{K}$, $\delta_{1}^{K}$ and $\delta_{2}^{K}$ are useful in studying Bergman spaces. The matching of $\tilde{\sigma}^{K}$ and the Bergman space is precisely the fact that they generate an automorphism-invariant metric on the domain; such metric may be viewed as natural, as it only depends on the complex structure, and not on the way it is located in $\C^{n}$. This motivates the following question.

\begin{question}
Let $X$ be a domain in $\C^{n}$, let $K$ be a sesqui-holomorphic non-degenerate positive semi-definite kernel and let $F$ be the corresponding RKHS. When is it possible to find a unitary-invariant Hermitean metric $\sigma$ on $F^{*}$ such that $\sigma^{K}$ is automorphism-invariant on $X$?
\end{question}

This question is of the most interest when $Aut\left(X\right)$ is "large", i.e. when $X$ is at least homogeneous.

Let us introduce an important class of operators on function spaces. Let $F$ be a RKHS over $X$ and let $\Phi:X\to X$ be such that $f\circ\Phi\in F$, for every $f\in F$. Then the operator $C_{\Phi}:F\to F$ defined by $\left[C_{\Phi}f\right]\left(x\right)=f\left(\Phi\left(x\right)\right)$ is called a \emph{composition operator} operator with \emph{symbol} $\Phi$. Obviously $C_{\Phi}$ is linear; by the Closed Graph Theorem it is continuous. Its adjoint $C_{\Phi}^{*}:F^{*}\to F^{*}$ "commutes" with $\kappa$, i.e. $C^{*}_{\Phi}\kappa\left(x\right)=\kappa\left(\Phi\left(x\right)\right)$, and so $C^{*}_{\Phi}$ can be viewed as a "linear extension" of $\Phi$. Also note that if $\Phi$ is a bijection and both $C_{\Phi}$ and $C_{\Phi^{-1}}$ are defined, then they are inverses of each other.

Assume that we were able to find a unitary-invariant metric $\sigma$ on $F^{*}$, which is invariant with respect to $C^{*}_{\Phi}$, for every automorphism $\Phi$ of $X$. Clearly, then $\sigma^{K}$ is invariant with respect to all of the automorphisms and our goal is accomplished. However, the following result shows that this can only happen in the trivial case.

\begin{proposition}\label{no}
Let $X$ be a weakly symmetric domain in $\C^{n}$ and let $F$ be a RKHS of holomorphic functions on $X$. Assume that there is a unitary-invariant metric $\sigma$ on $F^{*}$, which is invariant with respect to $C^{*}_{\Phi}$, for every automorphism $\Phi$ of $X$. Then $F$ consists of constant functions.
\end{proposition}

The fact follows immediately from part (iv) of Theorem \ref{rm1} and the following lemma.

\begin{lemma}\label{no1}
Let $X$ be a weakly symmetric domain in $\C^{n}$ and let $F$ be a RKHS of holomorphic functions on $X$ such that $C^{*}_{\Phi}$ is a congruency for every $\Phi\in Aut\left(X\right)$. Then $F$ is the space of constant functions.
\end{lemma}
\begin{proof}
Fix $x\in X$. Let $y\in X$ and let $\Phi\in Aut\left(X\right)$ be such that $\Phi\left(x\right)=y$ and $\Phi\left(y\right)=x$. Since $C^{*}_{\Phi}$ is a congruency we have that $$\overline{K\left(y,x\right)}=K\left(x,y\right)=K\left(\Phi\left(y\right),\Phi\left(x\right)\right)=\left\|C^{*}_{\Phi}\right\|^{2}K\left(y,x\right),$$ and so $K\left(y,x\right)^{2}\in\R$. Since $y$ was chosen arbitrarily we conclude that $K^{2}_{x}$ is a real-valued function. Combining this assertion with the fact that $K_{x}\in F$ is holomorphic, we conclude that $K_{x}$ is a real constant. Since $x$ was also chosen arbitrarily, we get that $K$ is a real constant, and so $F$ is the space of constant functions.
\end{proof}

It is easy to see that if $\Phi\in Aut\left(X\right)$, then both $C_{\Phi}$ and $C_{\Phi^{-1}}$ are congruencies if and only if both $C^{*}_{\Phi}$ and $C^{*}_{\Phi^{-1}}$ are congruencies. Thus as a byproduct we obtain the following result.

\begin{corollary}
Let $X$ be a weakly symmetric domain in $\C^{n}$. Then the only RKHS of holomorphic functions on $X$ such that $C_{\Phi}$ is a congruency for every $\Phi\in Aut\left(X\right)$, is the space of constant functions.
\end{corollary}

Proposition \ref{no} tells that no unitary-invariant Hermitean metric on the dual of a RKHS over $X$ can be invariant with respect to the action of $Aut\left(X\right)$ via the adjoints of the composition operators. However, the pull-back of such metric still can be automorphism-invariant (see more details in the next section).

\section{Projectively Invariant Kernels}\label{proj}

\textbf{Projective invariance.} Let us extend our collection of operators on function spaces beyond the composition operators. Let $F$ be a RKHS over a set $X$ and let $\omega: X\to\C$ be such that $\omega f\in F$, for every $f\in F$. Then a \emph{multiplication operator} with \emph{weight} $\omega$ is defined by $\left[M_{\omega}f\right]\left(x\right)=\omega\left(x\right)f\left(x\right)$.

Finally, if $\Phi:X\to X$ and $\omega: X\to\C$ are such that $\omega\cdot \left(f\circ\Phi\right)\in F$, for every $f\in F$, then a \emph{weighted composition operator} operator with \emph{symbol} $\Phi$ and \emph{weight} $\omega$ is defined by $\left[W_{\omega,\Phi}f\right]\left(x\right)=\omega\left(x\right)f\left(\Phi\left(x\right)\right)$. Note that $W_{1,\Phi}=C_{\Phi}$ and $W_{\omega,Id}=M_{\omega}$. Again, these operators are linear and continuous once they are defined (which is usually not trivial to determine for given $F$, $\Phi$ and $\omega$). We also have that $M^{*}_{\omega}\kappa\left(x\right)=\omega\left(x\right)\kappa\left(x\right)$ and $W^{*}_{\omega,\Phi}\kappa\left(x\right)=\omega\left(x\right)\kappa\left(\Phi\left(x\right)\right)$.

We will need the following result (we use the approach from \cite{le} and \cite{nz}).

\begin{proposition}\label{wu}
Let $F$ be a RKHS with a non-degenerate kernel $K$ and let $\Phi:X\to X$ be a surjection. Then $W_{\omega,\Phi}$ is unitary if and only if $\omega$ rescales $K\circ\Phi$ to $K$.
\end{proposition}

The proof follows from combining the two parts of the following lemma.

\begin{lemma} Let $F$ be a RKHS with a non-degenerate kernel $K$.
\item[(i)] Let $\lambda:X\to F^{*}$. There is a co-isometry $T:F\to F$ and $\omega:X\to\Cp$ such that $\omega\lambda=T^{*}\kappa$ if and only if $L\left(x,y\right)=\left<\lambda\left(x\right),\lambda\left(y\right)\right>_{F^{*}}$ is a rescalling of $K$.
\item[(ii)] If $\omega:X\to\Cp$ and $\Phi:X\to X$ is a surjection, then $W_{\omega,\Phi}$ is injective as long as it is defined.
\end{lemma}
\begin{proof}
(i): Necessity:
\begin{eqnarray*}
K\left(x,y\right)&=&\left<\kappa\left(x\right),\kappa\left(y\right)\right>_{F^{*}}=\left<T^{*}\kappa\left(x\right),T^{*}\kappa\left(y\right)\right>_{F^{*}}\\
&=&\left<\omega\left(x\right)\lambda\left(x\right),\omega\left(y\right)\lambda\left(y\right)\right>_{F^{*}}=\omega\left(x\right)\overline{\omega\left(y\right)}L\left(x,y\right).
\end{eqnarray*}

Sufficiency: Assume that $\omega:X\to\Cp$ rescales $L$ to $K$. Let $\iota\left(x\right)=\omega\left(x\right)\lambda\left(x\right)$. By Moore-Aronszajn Theorem, there is a unitary operator $S$ from $F^{*}$ onto\\ $\overline{\spa\iota\left(X\right)}\subset F^{*}$ such that $\iota\left(x\right)=S\kappa\left(x\right)$. Then $T=S^{*}$ is the required co-isometry.

(ii): The statement is equivalent to the fact that $W=W^{*}_{\omega,\Phi}\kappa\left(x\right)$ has a dense range. For every $x\in X$ there is $y\in x$, such that $\Phi\left(y\right)=x$, and so \\ $\kappa\left(x\right)=\frac{1}{\omega\left(y\right)}W\kappa\left(y\right)$. Thus the span of $W\kappa\left(X\right)$ contains $\kappa\left(X\right)$ and since the latter set spans $F^{*}$, we conclude that $W$ has a dense image.
\end{proof}

Let $X$ be a domain in $\C^{n}$ and let $L$ be a sesqui-holomorphic non-degenerate function on $X$. We will call $L$ \emph{projectively invariant} if $L\circ\Phi$ is a rescalling of $L$ for every $\Phi\in Aut\left(X\right)$. A \emph{multiplier} of a projectively invariant function $L$ is a function $\omega^{L}:Aut\left(X\right)\times X\to \Cp$, such that $\omega_{\Phi}^{L}=\omega^{L}\left(\Phi,\cdot\right)$ rescales $L\circ\Phi$ to $L$, for every $\Phi\in Aut\left(X\right)$. It is easy to see that if both $L\circ\Phi$ and $L\circ\Psi$ are rescallings of $L$, for $\Phi,\Psi\in Aut\left(X\right)$, then $L\circ\Phi\circ\Psi$ is also a rescalling of $L$, as well as $L\circ\Phi^{-1}$. Moreover, using part (iii) of Theorem \ref{rrig} and continuity of $L$ we can see that the automorphisms $\Phi$, such that $L\circ\Phi$ is a rescalling of $L$, form a pointwise closed subset of $Aut\left(X\right)$. Thus, in order to ensure that $L$ is projectively invariant, $L\circ\Phi$ has to be a rescalling of $L$ for every $\Phi$ in some collection of automorphisms of $X$ that generates a pointwise dense subgroup of $Aut\left(X\right)$. For example, if $X$ is a symmetric domain, then this collection can be chosen to be the union of the set of all holomorphic involutions on $X$ and the isotropy group at some $x\in X$.

It is also easy to see that if $L$ and $M$ are projectively invariant, then so is $LM$; from part (iii) of Theorem \ref{rrig}, if $L^{s}$ is well-defined for some $s\in\R$, it is also projectively invariant. A pointwise limit of projectively invariant functions is projectively invariant, although it may loose sesqui-holomorphicity; if the limit is with respect to the compact-open topology, sesqui-holomorphicity is preserved due to Weierstrass theorem. Finally, observe that any kernel of the form $L=h\otimes\overline{h}$, $h\in\Ho\left(X\right)$, is trivially projectively invariant.

It is also worth mentioning that if $L$ is a projectively invariant function on a homogeneous domain, and there is $y\in X$ such that $L\left(\cdot,y\right)$ does not vanish, then $L$ does not vanish. If $X$ is also bounded, then $L^{s}$ is well-defined since bounded homogeneous domains are simply-connected (see \cite[Corollary 1.10]{xu}).

Assume that $K$ is a sesqui-holomorphic non-degenerate positive semi-definite kernel on $X$ and $\tilde{\sigma}$ is the Fubini-Study metric. The map $\Phi\in Aut\left(X\right)$ is an isometry with respect to $\tilde{\sigma}^{K}$ if and only if the pull-back of $\tilde{\sigma}$ via $\kappa$ and via $\kappa\circ\Phi$ is the same, which is equivalent to $\tilde{\sigma}^{K}=\tilde{\sigma}^{K\circ\Phi}$. Due to Corollary \ref{caldelt}, the last condition means that $K$ and $K\circ\Phi$ are rescallings, i.e. there is a function $\omega^{K}_{\Phi}:X\to\Cp$, that rescales $K$ to $K\circ\Phi$, which by the proposition above happens if and only if $W_{\omega^{K}_{\Phi},\Phi}$ is unitary. Let us summarize in the following two results.

\begin{theorem}\label{prig}
The following are equivalent:
\item[(i)] $K$ is projectively invariant;
\item[(ii)] For every $\Phi\in Aut\left(X\right)$ there is a function $\omega_{\Phi}^{K}:X\to\C$, such that $W_{\omega^{K}_{\Phi},\Phi}$ is unitary;
\item[(iii)] Each (any) of $\tilde{\sigma}^{K}$, $\delta_{1}^{K}$ and $\delta_{2}^{K}$ is automorphism-invariant;
\item[(iv)] For every $\Phi\in Aut\left(X\right)$ there is $y\in X$ and an open $U\subset X$ such that for any $x\in U$ we have
$$\frac{\left|K\left(\Phi\left(x\right),\Phi\left(y\right)\right)\right|^{2}}{K\left(\Phi\left(x\right),\Phi\left(x\right)\right)K\left(\Phi\left(y\right),\Phi\left(y\right)\right)}=\frac{\left|K\left(x,y\right)\right|^{2}}{K\left(x,x\right)K\left(y,y\right)}.$$
\end{theorem}

\begin{proposition}
The product of projectively invariant kernels is projectively invariant. Any rescalling of a projectively invariant kernel is projectively invariant. A power of a projectively invariant kernel is projectively invariant, provided that it is well-defined and positive semi-definite. A compact-open limit of projectively invariant kernels is projectively invariant.
\end{proposition}

Now the term "projective invariance" is justified, because RKHS's with a projective invariant kernel admit a projective representation of $Aut\left(X\right)$ (see \cite{pasc}). The automorphism-invariance of the Bergman metric follows from Theorem \ref{prig}.

\begin{example}
Let $X$ be a bounded domain in $\C^{n}$ and let $\Phi$ be an automorphism $X$. Let
$J_{\Phi } \left(x\right)$ be the complex Jacobian of $\Phi$ at $x$; then $\left|J_{\Phi } \left(x\right)\right|^{2} $ is the real Jacobian. Hence, for each $g\in A^{2} \left(X\right)$, using the change of variables in the integration, we have
\begin{eqnarray*}
\left\| g\right\|^{2} _{A^{2} \left(X\right)} &=&\int_{X}\left|g\left(y\right)\right|^{2} d \mu\left(y\right)=\int _{X}\left|g\circ \Phi \left(x\right)\right|^{2} \left|J_{\Phi } \left(x\right)\right|^{2} d\mu\left(x\right) \\&=&\left\| J_{\Phi } \left( g\circ \Phi \right) \right\|^{2} _{A^{2} \left(X\right)}=\left\| W_{J_{\Phi },\Phi}f \right\|^{2} _{A^{2} \left(X\right)} ,
\end{eqnarray*}
thus $W_{J_{\Phi },\Phi}$ is an isometry of $A^{2} \left(X\right)$. This isometry is invertible, namely $W_{J_{\Phi^{-1} },\Phi^{-1}}=W_{J_{\Phi },\Phi}^{-1}$. Hence, $W_{J_{\Phi },\Phi}$ is a unitary on $F=A^{2} \left(X\right)$. Since $\Phi$ was chosen arbitrarily, we get that the Bergman kernel is projectively invariant, and so the Bergman metric is automorphism-invariant.\qed
\end{example}
\

\textbf{Multipliers.} We would like to conclude the section with discussing multipliers. Let $L$ be a projectively invariant sesqui-holomorphic function. Then its multiplier $\omega^{L}$ is "almost unique" due to Remark \ref{rss}: it is determined up to multiplication with a function from $Aut\left(X\right)$ into $\T$, and so we will write "$\approx$" for equalities modulo this action. Let us mention some of the properties of $\omega^{L}$. First of all, $\left|\omega^{L}_{\Phi}\right|^{2}=\frac{\widehat{L}}{\widehat{L}\circ\Phi}$; in particular if $K$ is the Bergman kernel, then $\left|J_{\Phi } \left(x\right)\right|^{2}=\frac{K\left(x,x\right)}{K\left(\Phi\left(x\right),\Phi\left(x\right)\right)}$, for any $x\in X$ and $\Phi\in Aut\left(X\right)$. If $L$ and $M$ are projectively invariant sesqui-holomorphic functions on $X$ and $\Phi,\Psi\in Aut\left(X\right)$, then $$\omega^{L}_{\Phi\circ\Psi}\approx\omega^{L}_{\Psi}\omega^{L}_{\Phi}\circ\Psi\mbox{ (the cocycle property), ~and~ }\omega^{LM}\approx\omega^{L}\omega^{M}. $$ In particular, in the case when $M=h\otimes\overline{h}$, we have that $\omega^{M}_{\Phi}\approx\frac{h}{h\circ\Phi}$, and so $\omega^{LM}_{\Phi}\approx\frac{h}{h\circ\Phi}\omega^{L}_{\Phi}$.

The following result appears in \cite{bb} in the context of Bergman-like spaces.

\begin{proposition}\label{mrig}
Let $X$ be a homogeneous domain in $\C^{n}$ and let $L$ and $M$ be non-degenerate sesqui-holomorphic projectively invariant functions on $X$. If $\omega^{L}\approx\omega^{M}$, then there is $c>0$ such that $M=cL$.
\end{proposition}
\begin{proof}
Fix some $y\in X$. For every $x\in X$ there is $\Phi\in Aut\left(X\right)$, such that $\Phi\left(y\right)=x$, and so $$\frac{\widehat{L}\left(y\right)}{\widehat{L}\left(x\right)}=\frac{\widehat{L}\left(y\right)}{\widehat{L}\circ\Phi\left(y\right)}=\left|\omega^{L}_{\Phi}\left(y\right)\right|^{2}=\left|\omega^{M}_{\Phi}\left(y\right)\right|^{2}=\frac{\widehat{M}\left(y\right)}{\widehat{M}\circ\Phi\left(y\right)}=\frac{\widehat{M}\left(y\right)}{\widehat{M}\left(x\right)}.$$ Hence $\widehat{M}\left(x\right)=\frac{\widehat{M}\left(y\right)}{\widehat{L}\left(y\right)}\widehat{L}\left(x\right)$ and, by virtue of Theorem \ref{rrig}, the result follows for $c=\frac{\widehat{M}\left(y\right)}{\widehat{L}\left(y\right)}$.
\end{proof}

Of course, if one of the two functions is positive semi-definite, then so is the other one.

\begin{example}[Weighted Bergman Space]
Let $X$ be a bounded domain in $\C^{n}$. Analogously to the Bergman Space we can define a \emph{weighted Bergman Space} $A^{2}_{w}$, i.e. the space of all holomorphic functions, square-integrable with respect to the measure $wd\mu$, where $w$ is a positive continuous $\mu$-integrable function on $X$. It turns out that this space is also a RKHS. Let $K^{w}$ be its kernel. It was shown in \cite{kol} that if $\Phi\in Aut\left(X\right)$, then there is a corresponding weight if and only if $\frac{w\circ\Phi}{w}$ is a square of a modulus of holomorphic function; moreover $\left|\omega^{K^{w}}_{\Phi}\right|^{2}=\frac{w\circ\Phi}{w}\left|J_{\Phi}\right|^{2}$. In particular, if $w=\left|h\right|^{2}$, where $h$ is holomorphic, then $A^{2}_{w}=\left\{f\left|hf\in A^{2}\left(X\right)\right.\right\}$, with the corresponding norm. Then $K^{w}=\frac{1}{h}\otimes\frac{1}{\overline{h}}K$, and the formula for the multiplier agrees with the one in the beginning of the subsection.

Assume that $X$ is homogeneous and fix $y\in X$. If $x=\Phi\left(y\right)$ we have that $$\frac{\widehat{K^{w}}\left(y\right)}{\widehat{K^{w}}\left(x\right)}=\left|\omega^{K^{w}}_{\Phi}\left(y\right)\right|^{2}=\frac{w\circ\Phi\left(y\right)}{w\left(y\right)}\left|J_{\Phi}\left(y\right)\right|^{2}=\frac{w\left(x\right)}{w\left(y\right)}\frac{\widehat{K}\left(y\right)}{\widehat{K}\left(x\right)},$$
where $K$ is the usual Bergman Kernel. Thus $\widehat{K^{w}}=c\frac{\widehat{K}}{w}$, for some $c>0$. This equality gives certain restriction on $w$ and at the same time "almost" determines $K^{w}$. In particular, if we know that there is a sesqui-holomorphic function $L$ such that $w=\widehat{L}$, then $K^{w}=cKL^{-1}$. For example, if $w=\widehat{K}^{-\alpha}$ is well-defined and integrable for $\alpha\in\R$, then $K^{w}=cK^{\alpha+1}$, and so by the observation before this example, $K^{\alpha+1}$ is a positive semi-definite kernel. Moreover, it is the unique kernel up to a positive multiple, with multiplier $J_{\Phi}^{\frac{\alpha+1}{2}}$. In fact, we have arrived at the same conclusions as in \cite{arazy}, using only elementary arguments (for the developments of these considerations see the introduction of \cite{bb}).\qed
\end{example}

\section{Supplements and questions}\label{sup}

\textbf{Remarks about Proposition \ref{mrig}.} The reader may have an impression that the assumption of homogeneity of $X$ in the proposition is too strong. Indeed, we could only assume that there is $x$ such that the orbit $Y= Aut\left(X\right)\cdot x$ is somewhere dense, i.e. $\Int\overline{Y}\ne\varnothing$. However, this condition implies weak symmetry, at least if $X$ is bounded. In order to prove this fact we first show that $Y$ is closed. Assume that $\left\{y_{k}\right\}_{k=1}^{\8}\subset Y$ converges to $y\in X$ and let $\Phi_{k}\in Aut\left(X\right)$ be such that $y_{k}=\Phi_{k}\left(x\right)$. Since $Aut\left(X\right)$ is relatively compact in $\Ho\left(X\right)^{n}=\Ho\left(X,\C^{n}\right)$ from Montel's theorem, we may assume that $\Phi_{k}\to\Phi\in\Ho\left(X\right)^{n}$. Of course then $\Phi\left(x\right)=y\in X$, but then from Cartan's theorem (see \cite[Theorem 1.10.7]{krantz}), $\Phi\in Aut\left(X\right)$, and so $y\in Y$. Now if $Y$ is a neighborhood of some point, from the (topological) homogeneity of $Y$ in $X$ it follows that $Y$ is a neighborhood of every of its points. Thus, $Y$ is open, and since $X$ is connected, its only closed and open subset is $X$ itself. Thus $Y=X$, and so $X$ is homogeneous.

Nevertheless, in the proof of the proposition we used very little information about $\omega$. Hence, it is natural to ask if a similar rigidity holds if we drop the sesqui-holomorphicity (analytic condition) and add positive semi-definiteness (algebraic/geometric condition).

\begin{question}\label{qq}
Let $X$ be a set (topological space) and let $K$ and $L$ be non-degenerate positive semi-definite (continuous) kernels on $X$. Assume that a (topological) group acts on $X$ transitively (and continuously) and both $K$ and $L$ are projectively invariant with respect to this action. Is it true that if they have equal multipliers, then there is $c>0$ such that $L=cK$?
\end{question}

Let $L$ and $M$ be as in Proposition \ref{mrig}. According to this proposition, if we can find a $h\in\Ho\left(X\right)$, such that $\omega^{M}_{\Phi}\approx\frac{h}{h\circ\Phi}\omega^{L}_{\Phi}$, then there is $c>0$, such that $M=ch\otimes\overline{h}L$. The question is if we can characterize $L$ and $M$ for which such $h$ exists. The quotient $\omega=\frac{\omega^{M}}{\omega^{L}}$ satisfies the cocycle property $\omega_{\Phi\circ\Psi}\approx\omega_{\Psi}\omega_{\Phi}\circ\Psi$, but for $\frac{h}{h\circ\Phi}$ the sign "$\approx$" turns into the identity, and so we need to find a characterization of the cocycles, for which this strict condition holds. For such "strict" cocycles one can show that if there is $y\in X$ such that $\omega\left(y\right)=1$, for any $\Phi$, which is an isotropy in $y$, then there is some function $h:X\to\Cp$, such that $\omega=\frac{h}{h\circ\Phi}$. However, the problem remains if $h$ is automatically holomorphic. Thus we can state a question, which lies in the realm of the elementary complex analysis.

\begin{question}
Let $X$ be a homogeneous domain in $\C^{n}$ and let $h:X\to\Cp$ be such that $\frac{h}{h\circ\Phi}$ is holomorphic, for any $\Phi\in Aut\left(X\right)$. Does it follow that $h$ is holomorphic?
\end{question}
\

\textbf{Remarks about Lemma \ref{no1}.} It is possible to show that a bounded domain is weakly symmetric (symmetric) whenever there is $x\in X$ and a somewhere dense set $Y\subset X$ such that for any $y\in Y$ there is (an involution) $\Phi\in Aut\left(X\right)$ such that $\Phi\left(x\right)=y$ and $\Phi\left(y\right)=x$. The proofs are very similar to the one in the beginning of the preceding subsection, but one has to employ the fact that the evaluation $\left(\Phi,x\right)\to  \Phi\left(x\right)$ is continuous ($Aut\left(X\right)$ is a topological group) with respect to the compact-open topology. Thus, we cannot replace the weak symmetry assumption in the lemma with its "local" analogue.

Nevertheless, it would be interesting to generalize the lemma to homogeneous domains. In its proof we have observed that $C^{*}_{\Phi}$ is a congruency for every\\ $\Phi\in Aut\left(X\right)$ if and only if $K\left(\Phi\left(x\right),\Phi\left(y\right)\right)=\left\|C^{*}_{\Phi}\right\|^{2}K\left(x,y\right)$, for any $x,y\in X$ and $\Phi\in Aut\left(X\right)$. Define $\alpha:Aut\left(X\right)\to\R$ by $\alpha\left(\Phi\right)=2\log \left\|C^{*}_{\Phi}\right\|$. Clearly, if $C^{*}_{\Phi}$ is a congruency for every $\Phi\in Aut\left(X\right)$, then $\alpha$ is a homomorphism, and the equality $K\left(\Phi\left(x\right),\Phi\left(y\right)\right)=e^{\alpha\left(\Phi\right)}K\left(x,y\right)$ ensures that $\alpha$ is continuous on $Aut\left(X\right)$. If we furthermore assume that $X$ is bounded, then an isotropy group at any point of $X$ is compact (see \cite[Lemma 1.11]{xu}). Hence, $\alpha$ maps this group into a compact subgroup of $\R$. The only such subgroup is trivial, and so we get that any automorphism with a fixed point must induce a unitary composition operator. Then Lemma \ref{no1} will follow provided the answer to the following question (which is of its own interest) is affirmative.

\begin{question}
Let $X$ be a bounded homogeneous domain in $\C^{n}$ and let $F$ be a RKHS of holomorphic functions on $X$ with a kernel $K$ such that $C_{\Phi}$ is a unitary for every isotropy $\Phi$ at $x\in X$. Does it follow that $K_{x}$ is a constant?
\end{question}

Assume that $X$ is balanced, i.e. $0\in X$ and $e^{i\theta}X=X$, for any $\theta\in\R$. If $K$ is as in the question with $x=0$, then $K_{0}\left(e^{i\theta}z\right)=K_{0}\left(z\right)$, for any $\theta\in\R$. For $z=\left[z_{1},z_{2},...,z_{n}\right]\in \C^{n}$ and $I=\left[i_{1},i_{2},...,i_{n}\right]\in\N_{0}^{n}$ denote $z^{I}=z_{1}^{i_{1}}z_{2}^{i_{2}}...z_{n}^{i_{n}}$ and $\left|I\right|=i_{1}+i_{2}+...+i_{n}$. Consider the Taylor expansion
$$K\left(0,0\right)+\sum\limits_{\left|I\right|>0}a_{I}z^{I}=K_{0}\left(z\right)=K_{0}\left(e^{i\theta}z\right)=K\left(0,0\right)+\sum\limits_{\left|I\right|>0}a_{I}e^{i\theta\left|I\right|}z^{I},$$
and so $a_{I}=a_{I}e^{i\theta\left|I\right|}$, for any $\theta\in\R$. If $\theta\not\in\Q$, then $e^{i\theta\left|I\right|}\ne 1$, which forces $a_{I}=0$, whenever $\left|I\right|>0$. Thus $K_{0}\left(z\right)\equiv K\left(0,0\right)$.

Note that any bounded symmetric domain is biholomorphic to a balanced one.

Also note, that we have used the equality $K\left(\Phi\left(y\right),\Phi\left(z\right)\right)=K\left(y,z\right)$, for isotropies at $x$, only in the case when $z=x$, and we also haven't used the fact that $K$ is positive semi-definite. Consequently, there should be room for generalizations.\\

Finally let us ask if an abstract version of Lemma \ref{no1} holds.

\begin{question}
Let $X$ be a topological space and let $K$ be a separately continuous positive semi-definite non-degenerate kernel on $X$. Assume that $G$ is a group that acts transitively and continuously on $X$ and there is continuous homomorphism\\ $\alpha:G\to\R$, such that $K\left(g\cdot x,g\cdot y\right)=e^{\alpha\left(g\right)}K\left(x,y\right)$, for any $x,y\in X$ and $g\in G$. Does it follow that $K$ is a constant?
\end{question}
\

\textbf{Remarks on our setting.} Although the phase space of our functions almost everywhere in the text was a domain in $\C^{n}$, most of the results and the proofs are valid for complex manifolds. However, the examples with Bergman spaces and their variations depend on the Lebesgue measure, which is a concept not inherent to manifolds. We conclude this section with brief explanation of how to adapt Theorem \ref{main} for the case of real manifolds. First note, that we have to include symmetry of kernels in the Moore-Aronszajn theorem for reals, since there are non-symmetric positive semi-definite functions.

Let $X$ be a $\Co^{l}$ real manifold and let $F$ be a RKHS over $X$ with kernel $K$ such that the corresponding $\kappa$ is $\Co^{m}$-smooth, for $0<m\le l$. Analogously to the complex case, every function $f\in F$ is also $\Co^{m}$-smooth, since $f=f^{**}\circ\kappa$; from the chain rule we get that $u f=\left<f,u\kappa\right>$, for any tangent vector $u$ at $x\in X$ of order at most $m$. Note that if all elements of $F$ are $\Co^{m}$-smooth, then $\kappa$ is $\Co^{m-1}$-smooth (see \cite{ger}). Consequently, if $X$ is a $\Co^{\8}$ manifold, then $\kappa$ is $\Co^{\8}$ if and only if all elements of $F$ are also  $\Co^{\8}$. It is easy to see that for $u\in T_{x}$, $x\in X$ we have that $u\widehat{K}=2uK\left(\cdot,x\right)$, but there is no analogue of the second formula in Corollary \ref{dpe2}. In fact, the expression there is not even well-defined for the real case. However, we can still use Proposition \ref{dpe1} and state the real version of Theorem \ref{main}.

\begin{proposition}
Let $X$ be a smooth manifold and let $F$ be a RKHS over $X$ which consists of smooth functions. Assume that a Riemannian metric $\sigma^{K}$ on $X$ is a pull-back of a certain unitary-invariant metric on $F^{*}\backslash\left\{0\right\}$. Then there are smooth functions $\varphi,\psi:\left(0,+\8\right)\to \R$, such that for any $x\in X$ and $u,v\in T_{x}$ we have that
$$\sigma^{K}_{x}\left(u,v\right)=\varphi\left(\widehat{K}\left(x\right)\right)uv_{1}K+\frac{1}{4}\psi\left(\widehat{K}\left(x\right)\right)u\widehat{K}v\widehat{K},$$ where $v_{1}K$ is the application of $v$ to $K$ with respect to the first variable.
\end{proposition}

\section{Acknowledgements}

The author wants to thank: his supervisor Nina Zorboska for extensive help with the preparation of this paper, and also Alexandre Eremenko and Robert Bryant from the website \href{http://mathoverflow.net/}{MathOverflow}.

\end{document}